\documentclass{amsart}
\usepackage{tikz}
\usepackage{hyperref}   
\usepackage{amssymb,color}

\usepackage{mathtools}  
\mathtoolsset{showonlyrefs}
\newtheorem{corollary}{Corollary}

\theoremstyle{plain}
\newtheorem{theorem}{Theorem}[section]
\theoremstyle{plain}
\newtheorem{lemma}[theorem]{Lemma}
\theoremstyle{remark}

\theoremstyle{plain}

\let\ra\rightarrow

\begin{document}

\

\title{On the degree of a finite cover which fibers over circle}

\author{Inkang Kim}
\address{Korea Institute for Advanced Study, Seoul 02455, South Korea}
\email{inkang@kias.re.kr}
\author{Hongbin Sun}
\address{Department of mathematics, Rutgers University, Hill Center-Busch campus, 110 Frelinghuysen Road, Piscataway, NJ 08854, USA}
\email{hongbin.sun@rutgers.edu}

\begin{abstract}
  We give a lower bound for the degree of a finite cover of a hyperbolic 3-manifold which fibers over the circle, in terms of volume, the diameter of the manifold and other new invariants.
\end{abstract}

{\maketitle}


\footnotetext[3]{Research by Inkang Kim is partially supported by Grant
NRF-2019R1A2C1083865 and KIAS Individual Grant (MG031408). Research by Hongbin Sun is partially supported by Simons Collaboration Grants 615229.}

{\tableofcontents}

\section{Introduction}
Thurston \cite{Th} conjectured that any closed hyperbolic 3-manifold has a finite cover which fibers over the circle. This dubious conjecture is completely solved in the affirmative by Ian Agol \cite{Agol}, yet it is hard to determine how large the degree of this finite cover is.  From the proof of the virtual fibering conjecture (in a group theoretic setting), it is not easy to track down the degree of the finite cover.
Recently, Stern {\cite{stern-scalar-2020}} showed an interesting formula relating the
level set of harmonic functions to the circle and the scalar curvature of a 3-manifold. Combining these ideas, we give a lower bound for the degree of a finite cover which fibers over the circle.

For a closed hyperbolic 3-manifold, we define the genus of $M$ as
$$g(M)=\inf\{g\geq 2: \exists\ \text{continuous}\ f:S_g\ra M\ \text{with injective}\
f_*:\pi_1(S_g)\ra \pi_1(M)\}.$$ Here $S_g$ denotes  the closed surface of genus $g$.
For a closed hyperbolic 3-manifold, $g(M)<\infty$ holds due to \cite{KM}.

Suppose $M_\phi$ is a minimal degree fibered covering of $M$ and $S$ is a fiber surface of $M_{\phi}$ with minimal $|\chi(S)|$.
Then $M_\phi$ is obtained from $S\times \mathbb{R}$ by identifying $S\times\{t\}$ with $S\times\{t+1\}$ via a pseudo-Anosov map $\phi$. Then $W=W(M_\phi)$ is defined to be the width of this fundamental domain, i.e., the supremum of the shortest distance between the points in $S\times\{0\}$ and the points in $S\times\{1\}$, among all representatives of $S$. Then one can show that   $$k\geq \frac{2\pi W(2g(M)-2)}{3\operatorname{vol}(M)}.$$                                 
Since it is hard to express $W$ in terms of invariants of $M$, we use the diameter of $M$ to get a lower bound for the degree. 

\begin{theorem}
  \label{main}Let $M$ be a closed hyperbolic 3-manifold.
  Suppose that  $M'$ is a degree $k$ cover of $M$ which fibers over the circle with a fiber $S$.
  Furthermore, if  we assume that $\operatorname{inj}M>\epsilon$ and the diameter of $M$ is large, then there exists  $\delta=\delta(\epsilon, g(S))$ such that
  $$k\geq \frac{(\mathrm{diam}(M)-\delta)2\pi(2g(M)-2)}{3\mathrm{vol}(M)}.$$
 \end{theorem}
 
An alternative for $W$ could be $||\phi||$, the translation length of a hyperbolic isometry corresponding to $\phi$, 
  and for large $||\phi||$
  $$k\geq \frac{2\pi (||\phi||-\delta)(2g(M)-2)}{3\operatorname{vol}(M)}.$$
  This $||\phi||$ is also inaccessible unless we know
  $M_\phi$ explicitly.
  
One can introduce another invariant of $M$ as follows.  For a closed hyperbolic 3-manifold $M$, define the fiber genus of $M$
$$fg(M)=\inf\{g: \exists\ \text{fibered covering}\ M_\phi\ \text{ of {\it M} with a fiber of genus}\ g\}.$$ Then obviously $ fg(M)\geq g(M)$ and the inequalities above hold with $g(M)$ being replaced by $fg(M)$.
Hence, using Lemma \ref{betti} one can get a lower bound for the degree in terms of the first betti number $b_1(M)$ of $M$ as follows.
\begin{theorem}
Let $M$ be a closed hyperbolic 3-manifold and  let $M'$ be a degree $k$ cover of $M$ which fibers over the circle with a fiber $S$.  Furthermore, if  we assume that $\operatorname{inj}M>\epsilon$ and the diameter of $M$ is large, then there exists  $\delta=\delta(\epsilon, g(S))$ such that
$$k\geq \frac{2\pi (\mathrm{diam}(M)-\delta)( b_1(M)-3)}{3 \operatorname{vol}(M)}.$$
If $W$ is large enough compared to $\delta$, then there exists $C=C(W, \operatorname{inj}M, g(S))$ such that 
$$k\geq \frac{2C\pi\operatorname{syst}(M) (b_1(M)-3)}{3 \operatorname{vol}(M)},$$ where $\operatorname{syst}(M)$ is the systole of $M$.
\end{theorem}

\section{Formula for Laplacian of energy of harmonic maps to the circle}

{{For the reader's convenience, we collect some basic formulas for Laplacian of the
energy of harmonic maps to the circle,
{\cite{stern-scalar-2020}}. }}Let $u : N \to \mathbb{S}^1$ be a harmonic map
from a closed Riemannian 3-manifold to the unit circle.

Choose an orthonormal frame $e_1, e_2, e_3$ adapted to $\Sigma_{\theta} = u^{-
1} (\theta)$, so that $e_1, e_2$ are tangential to $\Sigma_{\theta}$, and $e_3
= \frac{\nabla u}{| \nabla u|}$. Let $R_{ij}$ denote the sectional curvature
of $N$ for the section $e_i \wedge e_j$. The symmetric quadratic tensor
$(h_{ij} = \langle D_{e_i} e_3, e_j \rangle)$ is the second fundamental form
$k_{\Sigma_{\theta}}$ for $\Sigma_{\theta}$. Note that $k_{\Sigma_{\theta}} =
(| \nabla u|^{- 1} D \mathrm{d} u) |_{\Sigma_{\theta}}$.

Then the Gauss equation gives
\[ K = R_{12} + h_{11} h_{22} - h_{12}^2, \]
 the scalar curvature $R_N$ of $N$ is
\[ R_N = 2 (R_{12} + R_{13} + R_{23}) \]
and
\[ \ensuremath{\operatorname{Ric}} (e_3, e_3) = R_{13} + R_{23} . \]
The mean curvature $H_{\Sigma_{\theta}}
=\ensuremath{\operatorname{tr}}k_{\Sigma_{\theta}}$ is $h_{11} + h_{22}$ and
the scalar curvature $R_{\Sigma_{\theta}}$ is $2 K$. Hence
\[ \ensuremath{\operatorname{Ric}} (e_3, e_3) = \frac{1}{2}  (R_N -
   R_{\Sigma_{\theta}} + H^2_{\Sigma_{\theta}} - |k_{\Sigma_{\theta}} |^2) .
\]
Using harmonicity of $u$, one can verify that
\[ | \nabla u|^2  (H^2_{\Sigma_{\theta}} - |k_{\Sigma_{\theta}} |^2) = 2 |
   \mathrm{d} |h||^2 - |Dh|^2, \]
and one can rewrite
\begin{align}
  & \ensuremath{\operatorname{Ric}} (\nabla u, \nabla u) \\
  = & | \nabla u|^2 \ensuremath{\operatorname{Ric}} (e_3, e_3) \\
  = & \frac{1}{2}  | \nabla u|^2  (R_{\Sigma_{\theta}} - R_{\Sigma_{\theta}})
  + \frac{1}{2}  (2| \mathrm{d} | \nabla u||^2 - |D \mathrm{d} u|^2) .
  \label{schoen yau} 
\end{align}
Then using the standard Bochner identity for $\mathrm{d} u$
\begin{equation}
  \Delta_g \frac{1}{2}  | \nabla u|^2 = |D \mathrm{d} u|^2
   +\ensuremath{\operatorname{Ric}} (\mathrm{d} u, \mathrm{d} u),
   \label{bochner harmonic one form}
\end{equation}
one can deduce the formula as in  equation (2) of Theorem 1.1 in {\cite{stern-scalar-2020}}
\begin{equation}
  2 \int_N \frac{| \mathrm{d} u|}{2} R_{\Sigma_{\theta}} = 4 \pi \int_{\theta}
  \chi (\Sigma_{\theta}) \mathrm{d} \theta \geqslant \int_N R_N  | \mathrm{d}
  u|. \label{average}
\end{equation}
Here the first equality follows from the
coarea formula and Gauss-Bonnet theorem. 
%

\section{Degree bound for finite cover which fibers over circle}

Let $M_{\phi}$ be a hyperbolic mapping torus of $S$ via a pseudo-Anosov map
$\phi$ and
\[ u : M_{\phi} = S \times [0, 1] / (x, 0) \sim (\phi (x), 1) \to [0, 1] / 0
   \sim 1 \]
the projection. On the infinite cyclic cover $S \times \mathbb{R}$, $\phi$
acts as a translation. Since $$\pi_1 (M_{\phi}) = \langle \pi_1 (S), t|t \gamma
t^{- 1} = \phi_{\ast} (\gamma), \gamma \in \pi_1 (S) \rangle,$$ $\phi$
corresponds to a hyperbolic isometry $t$, and we denote $\| \phi \|$ the
translation length of $t$ on $\mathbb{H}^3$. This is the hyperbolic
translation length of $\phi$ on the infinite cyclic cover $S \times
\mathbb{R}$. Hence $\| \phi \|$ is greater than the width of the fundamental domain of
$M_{\phi}$ on the cyclic cover $S \times \mathbb{R}$ where the $S\times 0$ and $S\times 1$
sides are identified by the action of $\phi$.

{If $\text{rank} H_2(M_\phi, \mathbb{Z})\geq 2$,  it is known that $M_\phi$ has infinitely many different fiberings over the circle, hence we choose a fibering such that $|\chi(S)|\geq 2$ is minimal. From now on, we assume that $M_\phi$ is fibered over the circle with minimum $|\chi(S)|$ and $\phi$ is a pseudo-Anosov map of $S$.}
\begin{figure}[h]
  \begin{center}
    \begin{tikzpicture}[x=1cm,y=1cm]

\begin{scope}[shift={(2,0)}, thick]
\clip(-1.8,-2)rectangle(3,2);
\draw (0,0) circle [x radius=0.6, y radius=1.3];
\end{scope}
\begin{scope}[shift={(2,0)}, thick]
\clip(-1.8,-2)rectangle(3,2);
\draw (0,0.6) circle [x radius=0.1, y radius=0.5];
\end{scope}
\begin{scope}[shift={(2,0)}, thick]
\clip(-1.8,-2)rectangle(3,2);
\draw (0,-0.6) circle [x radius=0.1, y radius=0.5];
\end{scope}
\begin{scope}[shift={(-3,0)}, thick]
\clip(-1.8,-2)rectangle(3,2);
\draw (0,0) circle [x radius=0.6, y radius=1.3];
\end{scope}

\begin{scope}[shift={(-3,0)}, thick]
\clip(-1.8,-2)rectangle(3,2);
\draw (0,0.6) circle [x radius=0.1, y radius=0.5];
\end{scope}
\begin{scope}[shift={(-3,0)}, thick]
\clip(-1.8,-2)rectangle(3,2);
\draw (0,-0.6) circle [x radius=0.1, y radius=0.5];
\end{scope}

\draw[thick] (-4,-1.3) -- (4,-1.3);
\draw[thick] (-4,1.3) -- (4,1.3);
\draw[thick] (-2, -3) -- (1.5, -3);
\draw[thick] (-2.4, 0) -- (1.4,0);
\draw[thick] (-2.44, -0.3) -- (1.4, -0.2);


\draw (-2, -2.8) node{{0}};
\draw (1.5, -2.8) node{{1}};
\draw(0.5, -2) node{{$u$}};
\draw(0, -2) node{{$\downarrow$}};
\draw (-0.5, -0.9) node{{$\Longrightarrow$}};
\draw (4, 0) node{{$S\times {\mathbb R}$}};
\draw (-0.47,-0.5) node{{$||\phi||$}};
\draw (-0.3,0.2) node{{W}};
\draw (-2.3,0.5) node{{$\rightarrow$}};
\draw (-1.85,0.5) node{{${\nabla u}$}};
\draw (-2.3,0.02) node{{$\neg$}};
\draw (1.28,0.04) node{{$\leftharpoondown$}};
\draw (-3, -1.6) node{{$u^{-1}(0)$}};
\draw (2, -1.6) node{{$u^{-1}(1)$}};
\draw (-5,0) node{{\tiny $-\infty$}};

\end{tikzpicture}
  \end{center}
  
  \
  \caption{Lifted map of $u$ to $S \times \mathbb{R} \rightarrow [0, 1]
  \subset \mathbb{R}$.}
\end{figure}
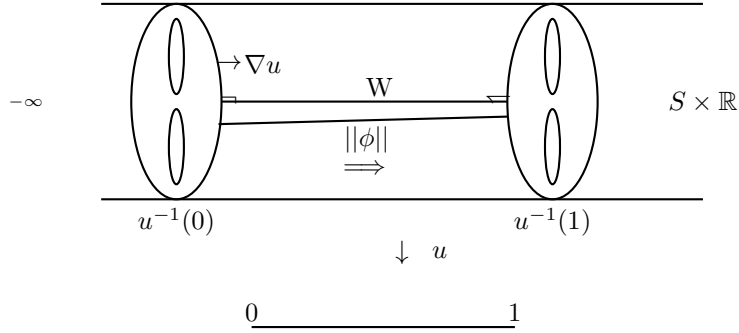

Choose an orthonormal basis $e_1, e_2$ tangent to $S = u^{- 1}
(\theta)$, $e_3$ such that $e_3 = \frac{\nabla u}{| \nabla u|}$. Then $|
\mathrm{d} u| = | \nabla u|$ and
$$\| \mathrm{d}
  u\|_{L^2}=\sqrt{\int_{M_\phi} |\nabla u|^2}.$$ Note that  for any $x\in S\times \{0\}$,  along the gradient flow starting from $x$
 in  $u^{-1}(0)$ to $u^{-1}(1)$, $u$ behaves like the projection from $[0, \text{length of the gradient flow}]$ to $[0,1]$. Hence $$|du|\leq \frac{1}{ \text{length of the gradient flow}}\leq \frac{1}{\text{W=width of the fundamental domain}}.$$ Here the width means the smallest distance between points in
  $S\times\{0\}$ and points in $S\times\{1\}$.
 Then $$\| \mathrm{d}
  u\|_{L^2}\leq \frac{\sqrt{\ensuremath{\operatorname{vol}} (M_{\phi})}}{(\text{W=width of the fundamental domain})}.$$ In general if we choose a different fibering than the minimal $|\chi(S)|$ fibering, this width can go to zero as $|\chi(S)|$ tends to infinity.

Let $u'$ be a harmonic map homotopic to $u$. For each regular value {{ $\theta \in
\mathbb{S}^1$}}, $u^{- 1} (\theta) = S$ and $u^{\prime - 1} (\theta) = \Sigma_{\theta}$
are homologous in $M_{\phi}$. Since $M_{\phi}$ is a mapping torus of $S$, the genus of
$\Sigma_{\theta}$ is at least the genus of  $S$, i.e. $-\chi(S)\leq -\chi(\Sigma_\theta)$. By Sard's theorem,  $-\chi(S)\leq -\chi(\Sigma_\theta)$ for almost all {{$\theta$ in $\mathbb{S}^1$}}. Hence by equation \eqref{average}, we get
\begin{equation}\label{genus}
  - 4 \pi \chi (S) \leqslant \| \mathrm{d} u' \|_{L^2} \|6\|_{_{L^2}}
  \leqslant 6 \sqrt{\ensuremath{\operatorname{vol}} (M_{\phi})} \| \mathrm{d}
  u\|_{L^2} = \frac{6\ensuremath{\operatorname{vol}} (M_{\phi})}{W}.
\end{equation}

Suppose $M_\phi \ra M$ is a $k$-fold covering map.  Then $\mathrm{vol}(M_\phi)=k \mathrm{vol}(M)$, 
and by equation \eqref{genus},
$$-4\pi\chi(S)\leq \frac{6 \ensuremath{\operatorname{vol}} (M_{\phi})}{W}
\leq \frac{6k \ensuremath{\operatorname{vol}} (M)}{W}.$$
This gives a lower bound
\begin{equation}\label{deg-lower}
k\geq \frac{2\pi W|\chi(S)|}{3\operatorname{vol}(M)}.
\end{equation}

{\bf Remark:} For a finite volume hyperbolic 3-manifold  $N$ with cusps, one can truncate the manifold along horotori to obtain a manifold $M$ with boundary.
Then the formula  reads as (see \cite{HS}):
$$\int_M R_M|du| + \int_{\partial M} H_{\partial M} |du|\leq 4\pi \int_{S^1} \chi(\Sigma_\theta).$$ By truncating the cusps deeper and deeper, we can make the area of the boundary tends to zero, hence the second term in the formula goes to zero, and finally we can assume that the same formula \eqref{average} holds even for cusped hyperbolic 3-manifolds.

\section{Better lower bound with injectivity radius lower bound}
Suppose that a closed hyperbolic 3-manifold $M$, whose finite cover $M_\phi$ is fibered over the circle with fiber $S$ with minimal $|\chi(S)|$ among possible different fiberings, has a  lower bound for the injectivity radius by $\epsilon$.
Take a pleated surface $f:X\ra M$ from a hyperbolic surface $X$ homeomorphic to $S$, and then take a lift
$$\tilde f: X\ra  M_\phi.$$ Then $\tilde f(X)$ is homotopic to $S$ in $M_\phi$.
Here we assume the injectivity radius of $M$ is at least $\epsilon$, the diameter of $\tilde f(X)$ is bounded above by some number $\delta$. Indeed $\delta$ is the upper bound of the diameter of hyperbolic surfaces whose injectivity radius is bounded below by $\epsilon$ and genus at least $g(S)$.
Consider a further lift $$f':X\ra \tilde M_\phi$$ where $\tilde M_\phi$ is the infinite cyclic cover over which $\phi$ acts as a translation with the translation length $||\phi||$. Then $\phi$ identifies $f'(X)$ to its another lift $\phi\circ f'$, which implies that the diameter of $M_\phi$ is less than $W+\delta$ where $\delta=\delta(\epsilon, g(S))$ and $W$ is the width of the fundamental domain as in the previous section. Hence
$$\mathrm{diam}(M)\leq \mathrm{diam}(M_\phi)\leq W+\delta.$$
Finally we obtain
\begin{equation}\label{diam}
k\geq \frac{(\mathrm{diam}(M)-\delta)2\pi|\chi(S)|}{3\mathrm{vol}(M)}.
\end{equation}

If $M_\phi$ is a finite cover of $M$, which fibers over the circle with fiber $S$, then
$f:S\hookrightarrow M_\phi \ra M$ satisfies that $f_*$ is injective. Hence
$g(S)\geq g(M)$. Then the above inequality reads
$$k\geq \frac{(\mathrm{diam}(M)-\delta)2\pi(2g(M)-2)}{3\mathrm{vol}(M)}.$$ 

Once there is a lower bound for the injectivity radius and the genus of $S$ is fixed,
the diameter of $S$ has an upper bound by $\delta$, and since $\phi$ identifies $S\times \{0\}$ to $S\times \{1\}$, we obtain $||\phi||\leq W+\delta$ so that $W\geq ||\phi||-\delta$ for large $W$, hence we
get
\begin{equation}\label{hyp-trans}
k\geq \frac{2\pi (||\phi||-\delta)(2g(M)-2)}{3\operatorname{vol}(M)}.
\end{equation}


Note that the lower bounds in the main theorem holds with $g(M)$ being replaced by $fg(M)$.
The possible advantage of using $fg(M)$ instead of $g(M)$ is as follows.
\begin{lemma}\label{betti}The fiber genus of $M$ satisfies
$fg(M)\geq \frac{b_1(M)-1}{2}$.
\end{lemma}
\begin{proof}Note that $b_1(M_\phi)\leq  b_1(S)+1$ where $S$ is the fiber. Hence
$b_1(M)\leq b_1(M_\phi)\leq 2g(S)+1$, and the claim follows.
\end{proof}
Hence we obtain
\begin{corollary}Let $M$ be a closed hyperbolic 3-manifold and  $M'$ be a degree $k$ cover of $M$ which fibers over the circle. Then
$$k\geq \frac{2\pi W (b_1(M)-3)}{3 \operatorname{vol}(M)}.$$
\end{corollary}

\section{Intrinsic definition for W}
Note that $W=W(M)$ is defined to be the width of the fundamental domain of $M_\phi$ in
$S\times \mathbb R$ where $S$ is the minimal genus fiber of a minimal degree
cover which fibers over the circle.
What is a practical way to calculate $W$, or an intrinsic definition for $W$?

When $\operatorname{inj}M>\epsilon$ and $W$ is large enough compared to $\delta$, then $W\geq C ||\phi||$ where $C=C(W,\epsilon, g(S))$.
Hence $W\geq C \text{syst}(M)$ where $\text{syst}(M)$ is the systole of $M$.
In general, $C\ra 0$ as $g(S)\ra\infty$.
In summary, 
\begin{corollary}
When there is a lower bound for the injectivity radius of $M$, we obtain for some $C=C(W, \operatorname{inj}M, g(S))$
$$k\geq \frac{2C\pi \operatorname{syst}(M) (b_1(M)-3)}{3 \operatorname{vol}(M)}.$$
\end{corollary}

\section{ Some examples}
\subsection{Examples for (fiber) genus}
In this section, we construct a sequence of hyperbolic 3-manifolds whose (fiber) genus tends to infinity.
By Lemma \ref{betti}, $fg(M)\geq \frac{b_1(M)-1}{2}$, hence any sequence of hyperbolic 3-manifolds $M_n$ with $b_1(M_n)\ra\infty$ will have $fg(M_n)\ra\infty$.

The real question is whether there exists a sequence of hyperbolic 3-manifolds whose genus tends to infinity.
This question has to do with the injectivity radius of the manifold.
Let $f:S_g\ra M$ be a continuous map with injective $f_*:\pi_1(S_g)\ra\pi_1(M)$. Realize the surface by a pleated surface $p:(S,h)\ra M$ within the homotopy class.
If $r$ is the injectivity radius of $M$, then the injectivity radius of $(S,h)$ is greater than or equal to $r$. Hence
$$ \mathrm{Area}(D_r)\leq \mathrm{Area}(S,h)=2\pi(2g-2),$$ where $D_r$ is a hyperbolic disk of radius $r$.
This shows that the genus of such a surface tends to infinity as the injectivity radius of $M$ tends to infinity. Hence, once we take a sequence of hyperbolic 3-manifolds $M_n$ whose injectivity radius tends to infinity, the genus $g(M_n)\ra\infty$.

\subsection{Examples for $fg(M)-g(M)\ra\infty$}
Let $M$ be a closed hyperbolic 3-manifold that contains an embedded non-separating $\pi_1$-injective surface $i:S\to M$ such that $$\mathrm{dim} (H_1(M,\mathbb R)/i_*(H_1(S,\mathbb R)))$$ is at least 2. Such a pair $(M,S)$ can be constructed as follows.  Agol \cite{Agol} proved that the group of any closed hyperbolic 3-manifold $N$ virtually surjects to $\mathbb Z*\mathbb Z$.
Realize the homomorphism by a map $f:M\ra S^1 \vee S^1$ for some finite cover $M$ of $N$. Then take $S$ to be the preimage of a non-singular point.

Let $M_n$ be the $n$-sheeted cyclic cover of $M$ obtained by taking $n$ copies of $M\setminus S$ and pasting them together. Then one can show that $b_1(M_n)\ra\infty$, hence $fg(M_n)\ra\infty$, but $g(M_n)\leq g(S)$.

\subsection{Examples for Dehn fillings}
A natural way to obtain a family of closed hyperbolic $3$-manifolds is applying Dehn fillings to a fixed hyperbolic knot/link complement. 

At first, by \cite{KW}, any cusped hyperbolic $3$-manifold $M$ contains a $\pi_1$-injective quasi-Fuchsian closed subsurface $S\looparrowright M$. If $M$ has only one cusp (e.g. a hyperbolic knot complement), then by \cite{Ba}, for all but finitely many hyperbolic Dehn fillings of $M$, $S$ is still $\pi_1$-injective in the filled manifold, so $g(\cdot)$ is bounded above by the genus of $S$. In this case, $g(\cdot)$ is bounded above for all hyperbolic Dehn fillings of $M$. 

If $M$ has more than one cusps, by \cite{Ba}, $g(\cdot)$ is bounded by the genus of $S$ for sufficiently long Dehn fillings of $M$. It is not clear whether $g(\cdot)$ is bounded above for all hyperbolic fillings of $M$ in this case.


\vskip .1 in
{\bf Acknowledgement} The  first author would like to thank Ian Agol for answering questions about the virtual Haken conjecture.

\end{document}